\documentclass[12pt]{amsart}
\usepackage{amssymb,amsfonts,dsfont,ulem,xcolor,fullpage}

\begin{document}
\newtheorem{thm}{Theorem}[section]
\newtheorem{defn}[thm]{Definition}
\newtheorem{Conj}[thm]{Conjecture}
\newtheorem{ex}[thm]{Example}
\newtheorem{exs}[thm]{Examples}
\newtheorem{rem}[thm]{Remark}
\newtheorem{prop}[thm]{Proposition}
\newtheorem{lemma}[thm]{Lemma}
\newtheorem{cor}[thm]{Corollary}
\newtheorem{obs}[thm]{Observation}
\newtheorem{ax}[thm]{Axiom}
\newtheorem{ques}[thm]{Question}
\newtheorem{nota}[thm]{Notation}

\title{Locally compact,  $\w_1$-compact spaces}

\author{Peter Nyikos}
\address[Peter Nyikos]{Department of Mathematics, University of South Carolina, Columbia, SC 29208}
\email{nyikos@math.sc.edu}

\author{Lyubomyr Zdomskyy}
\address{Universit\"at Wien, Institut f\"ur Mathematik, Kurt G\"odel Research Center,
Kolingasse 14--16, 1090 Wien.}
\email{lzdomsky@gmail.com}
\urladdr{http://www.logic.univie.ac.at/$\tilde{\ }$lzdomsky/}

\def \frak{\mathfrak}
\def \cal{\mathcal}
\def \Bbb{\mathbb}
\def \a{\alpha}
\def \b{\frak b}
\def \B{\beta}
\def \c{\frak c}
\def \g{\gamma}
\def \d{\delta}
\def \l{\lambda}
\def \lg{\langle}
\def \rg{\rangle}
\def \s{\sigma}
\def \k{\kappa}
\def \w{\omega}
\def \z{\zeta}
\def \be{\beta}
\def \A{\aleph}
\def \L{\cal L}
\def \La{\Lambda}
\def \R{\mathbb R}
\def \P{\mathbb P}
\def \Q{\mathbb Q}
\def \G{\cal G}
\def \U{\cal U}
\def \V{\cal V}
\def \from{\leftarrow}
\def \flag{\restriction}
\def \nix{\emptyset}

\def\ddemo{\quad $\square$ \medskip}
\def \ob{\overline}
\def \ol{\overline}
\def \down{\downarrow}
\def \sm{\setminus}

\subjclass[2010]{Primary 54D15, 54D45.  Secondary 03E35, 54C10, 54D35}


\thanks{The second author would
like to thank  the Austrian Science Fund FWF (Grants I 1209-N25,  I 2374-N35 and I 3709-N35)
 for generous support for this research.}

\begin{abstract}
An $\w_1$-compact space is a space in which every closed discrete subspace
is countable.  We give various general conditions under which  a locally compact, $\w_1$-compact
space is {\it $\s$-countably compact, i.e.,} the union of countably many countably
compact spaces. These conditions involve very elementary properties.
\medskip

Many results shown here   
are independent of the usual (ZFC) axioms of set theory, and the
consistency of some  may involve large cardinals.
For example, it is independent of the ZFC axioms whether every locally
compact, $\w_1$-compact space of cardinality $\aleph_1$
is $\s$-countably compact. Whether $\aleph_1$ can be replaced with
$\aleph_2$ is a difficult unsolved problem. Modulo large
cardinals, it is also ZFC-independent whether every hereditarily normal,
or every monotonically normal, locally
compact, $\w_1$-compact space is $\s$-countably compact.
\medskip

As a result, it is also ZFC-independent whether there is a locally
compact, $\w_1$-compact Dowker space of cardinality $\A_1$, or
one that does not contain both an uncountable closed discrete
subspace and a copy of the ordinal space $\w_1$.
\medskip

\bigskip

Set theoretic tools used for the consistency results
 include the existence of a Souslin tree, the
Proper Forcing Axiom (PFA), and models
generically referred to as ``MM(S)[S]''. Most of the work is done by
the P-Ideal Dichotomy (PID) axiom, which holds in  the latter
two cases,
and which requires no large cardinal axioms when directly applied to topological
spaces of cardinality $\A_1$, as it is in several theorems.
\end{abstract}

\maketitle

\medskip

\section{\label{intro} Introduction}

A {\it space of countable extent}, also called an  {\it $\w_1$-compact
space}, is one in which every closed discrete subspace is countable.
Obvious examples of  $\w_1$-compact spaces are countably compact spaces (because
in them every closed discrete subspace is finite), and
{\it $\s$-countably compact spaces}, {\it i.e.,} the union of
countably many countably compact spaces. On the
other hand,
an elementary application of the Baire Category Theorem shows that
the space of irrational numbers with the usual topology is not
$\s$-countably compact, but like every other separable metrizable
space, it is $\w_1$-compact.
\medskip

The situation is very different when it comes to locally
compact spaces.
In an earlier version of this paper due to the first author, he asked:

\begin{ques} \label{1.1}
Is there a ZFC example of a locally compact, $\w_1$-compact space
of cardinality $\A_1$ that is not $\s$-countably compact? one that is
normal? \end{ques}

Here too, local compactness makes a big difference: without it,
the space of irrationals is a counterexample under CH, while ZFC is
enough to show any cardinality $\A_1$ subset of a Bernstein set
is a counterexample.
\bigskip

As it is, the second author showed that the answer to  Question \ref{1.1}
is negative; see Section 2. On the other hand, both the Kunen line
and a Souslin tree with the usual topology are consistent locally compact
normal examples for Question \ref{1.1}. In the case of the Kunen line,
its hereditary separability clearly implies $\w_1$-compactness,
and its hereditary realcompactness implies that every countably compact
subspace is compact and therefore countable. The case of a Sousln tree will
be shown at the end of  Section 5. A third consistent example is given in
\cite{Nyclub}:

\begin{thm} \label{club} Assuming $\clubsuit$, there is
a locally compact, locally countable (hence first countable)
 $\w_1$-compact space of
cardinality $\aleph_1$ which is not $\s$-countably compact. \end{thm}

\bigskip

It is not known whether this example is normal. However, in Section 4 we will
construct an example with all the stated properties of Theorem \ref{club} and
which is monotonically normal.

\medskip

Monotonically normal spaces are, informally speaking, ``uniformly normal''
[see Definition \ref{monO} below]. They are hereditarily normal, and this
theorem gives another independence result when combined with:

\begin{thm}\label{mm} In MM(S)[S] models,
every hereditarily normal,
locally compact, $\w_1$-compact space is $\s$-countably compact.
\end{thm}

 An even stronger theorem will be shown in Section \ref{T5}
which includes the statement that the PID is enough to show
Theorem \ref{mm} for monotonically normal spaces. It also puts some limitations
on what kinds of Dowker spaces (that is, normal spaces $X$
such that $X \times [0,1]$ is not normal) are possible
if one only assumes the usual (ZFC) axioms of set theory.

\bigskip

\medskip

In the light of the negative answer to Question \ref{1.1},
it is natural to ask:
\bigskip

\noindent{\bf Problem 1.}
 {\it What is  the least cardinality of
 a locally compact, $\w_1$-compact space which is not
$\s$-countably compact? one that is normal?}
\medskip

The best ZFC result so far was shown by the first author \cite{N7}:

\begin{ex}\label{minb}There is a locally compact, normal,
$\w_1$-compact space of cardinality  $\frak b$ that is not $\s$-countably compact. \end{ex}

Previously, the best upper bound for the minimum was
$\frak c$, using one of E.K. van Douwen's ``honest submetrizable''
examples \cite{vD1}. However,  Example \ref{minb} still leaves a lot unsaid. See Section 5 for a discussion,
especially of the following ``echo'' of Question \ref{1.1}.

\medskip

\noindent
{\bf Problem 2.} {\it Is there a ZFC example of a  locally compact,
$\w_1$-compact space of cardinality $\A_2$ that is not
$\s$-countably compact? one that is normal?}
\bigskip

 The last section
gives more information about Example  \ref{minb}  and
about a related problem and result of Eric van Douwen,
under the assumption of $\b = \c$.
\medskip

In between,
Section 6 gives some interesting counterpoints to Problem 1 by discussing questions about the greatest cardinality
of a locally countable, normal, countably compact space. Local countability may seem like a very specialized property,
but it actually holds in most of our examples, including Example \ref{minb} under CH,
and it easily implies local compactness in countably compact spaces, and with it, first countability.
\bigskip

The individual sections are only loosely connected with each other, and each can
be read with minimal reliance on any of the others.
\medskip

All through this paper, ``space'' means ``Hausdorff topological space.''
All of the spaces described are locally compact, hence Tychonoff; and
all are also normal, except for a consistent example at the end.

\section{\label{A1}{The cardinality $\A_1$ case}}

 The P-Ideal Dichotomy (PID) plays a key role in this section and the following
one. It has to do with the following concept.

\begin{defn}\label{ideal}
\rm A family $\cal I\subset \mathcal P(X)$ is an \it{ideal on a set $X$},\ \rm
 if it is closed under finite unions, contains all singletons, and is closed under taking subsets of
its elements.
 An ideal $\mathcal I$ is a \it{P-ideal} \rm if
for every countable subfamily $\cal Q$ of $\cal I$, there exists $I \in \cal I$
such that $Q \subset^* I$ for every $Q \in \cal  Q$. Here $Q \subset^* I$
means that $Q \sm I$ is finite.\end{defn}
\medskip

The PID states that, for every P-ideal $\cal I$ of subsets of a set $X$,
either
\smallskip

(A) there is an uncountable $A \subset X$ such that every
countable subset of $A$ is in $\cal I$, or
\smallskip

(B) $X$ is the union of countably many sets $\{B_n : n \in \w\}$
such  that $B_n \cap I$ is finite for all $n$ and all $I \in \cal I.$
\medskip

The routine proofs of the  next lemma and theorem were given in \cite{EN}:

\begin{lemma} \label{a1} Let $X$ be a locally compact Hausdorff space.
The countable closed discrete subspaces of $X$ form a P-ideal if,
and only if, the extra point $\infty$ of
the one-point compactification $X + 1$ of $X$
is an $\a_1$-point; that is, whenever $\{\s_n : n \in \w\}$
is a countable family of sequences converging to $\infty$,
then there is a sequence $\s$ converging to $\infty$ such
that $ran(\s_n) \subseteq^* ran(\s)$ for all $n$.
\end{lemma}

The key is that an ordinary sequence in $X$
converges to the extra point of $X + 1$
if, and only if, its range is a closed discrete subspace of $X$.

\medskip

\begin{thm} \label{tri1} Assume the PID axiom. Let $X$ be a
locally compact space. Then at least one of the following is true:

$(1^-)$ $X$ is the union of countably many subspaces $Y_n$ such
that each sequence in $Y_n$ has a limit point in $X$.

$(2)$ $X$ has an uncountable closed discrete subspace

$(3^+)$ The extra point of $X + 1$ is not an $\alpha_1$-point.
\end{thm}

The key here is that (A) goes with (2), (B) goes with $(1^-$), and
$(3^+)$ is equivalent to the countable closed discrete subspaces failing
to form a P-ideal, by Lemma \ref{a1}.
\bigskip

The following is well known:

\begin{lemma} \label{b} If $X$ is a space of character $< \frak b$, then
every point of $X$ is an $\a_1$-point.\end{lemma}

We now have a negative answer to the second part of
Question \ref{1.1}.

\begin{thm} \label{Tri2} Assume the PID and $\frak b > \A_1$.
Then every locally compact, $\w_1$-compact, normal space
of cardinality $\A_1$ is
$\s$-countably compact. \end{thm}

\begin{proof} In a locally compact space, character $\le$ cardinality.
Lemmas \ref{a1} and \ref{b} and $\w_1$-compactness give us alternative
$(1^-)$ of Theorem \ref{tri1}. The closure of each $Y_n$ is easily
seen to be pseudocompact (i.e., every continuous real-valued
function is bounded). In
a normal space, every closed pseudocompact subspace is countably compact, cf.
 \cite[17J 3]{W}.
So the closures of the $Y_n$ witness that $X$ is $\s$-countably compact.
\end{proof}

\medskip

As shown in \cite{EN}, the hypothesis of normality in Theorem \ref{Tri2}
can be greatly weakened to ``Property wD''. Also, the proof of
Theorem \ref{Tri2} clearly extends to show that every normal,
locally compact, $\w_1$-compact space of cardinality (or even Lindel\"of number) $< \frak b$
is $\s$-countably compact. However, this may be a very limited improvement:
the PID implies $\frak b \le \aleph_2$. This is  a theorem of Todor\v{c}evi\'c,
whose proof may be found in \cite{MZ}.

\bigskip
The axioms used in Theorem \ref{Tri2} follow from the
Proper Forcing Axiom ($PFA$) and hold in $PFA(S)[S]$ models.
Each of these models is formed from a $PFA(S)$ model by forcing with
a coherent Souslin tree $S$ that is part of the definition
of what it means to be a  $PFA(S)$ model. The rest of the definition
states that
every proper poset $P$ that does not destroy $S$ when it
is forced with, has the following property. For every
family of $\le \A_1$ $\uparrow$-dense, $\uparrow$-open sets,
there is a $\downarrow$-closed, $\uparrow$-directed
subset of $P$ that meets them all.\footnote{The notation is as in \cite{F}. It is the ``Israeli" notation, whereby stronger conditions are larger.
 The topology to which
it refers is the one where each $p \in P$ has the one basic
neighborhood $\{t: t \ge p\}$.}
 The PFA is similarly
defined by omitting all mention of $S$. What remains is very similar
to the well-known definition of
Martin's Axiom $(MA)$; the only difference is that $MA$ uses
``ccc'' instead of ``proper.''
\bigskip

In this paper, we will use a slight abuse of language
with expressions like $PFA(S)[S]$ and $MM(S)[S]$ as though they
were axioms. The latter is defined like the former, but
with ``semi-proper'' in place of ``proper.''

\bigskip
For our negative answer to the first half of Question \ref{1.1}, we
needed a strengthening of $\frak b > \A_1$ which holds in models of
$PFA(S)[S]$ and $MM(S)[S]$.


\begin{lemma} \label{Z} Let $X$ be a $T_3$ space of weight
$ < \min\{\mathfrak b,\mathfrak s\}$,
 and let $Y \subset X$.
  Suppose that  no $Z\in [Y]^\w$
is closed discrete in $X$. Then there exists a
countably compact $Y'\subset X$ containing $Y$.
\end{lemma}

\begin{proof}
Using ideas from \cite{Boo74},
for every $Z\in [Y]^\w$ we can find $S_Z\in [Z]^\w$ which  converges
to some $x_Z\in X$. More precisely,
let $\mathcal Z$ be the family of infinite elements of $\{Z\cap U:U$ is a basic open subset of $X\}$.
Since $|\mathcal Z|<\mathfrak s$, $\mathcal Z$ is not a splitting subfamily of $[Z]^\w$, so there exists $S_Z \subset Z$ unsplit by $\mathcal Z$. Then $S_Z$
has at most one limit point in $X$: If $x_1$ and $x_2$ were two of them, pick open disjoint basic
neighbourhoods $U_1\ni x_1$ and $U_2\ni x_2$. Then both of the  sets
$U_1\cap S_Z=(U_1\cap Z)\cap S_Z$ and $U_2\cap S_Z=(U_2\cap Z)\cap S_Z$ are infinite, and hence each of $U_1\cap Z,U_2\cap Z\in\mathcal Z$  splits $S_Z$, a contradiction. Since $S_Z$ has a limit point $x_Z\in X$, it follows
that $S_Z$ converges to $x_Z$.

Let $Y' =Y\cup\{x_Z:Z\in [Y]^\w\}$.
{\it We claim that $Y'$ is countably compact.} Indeed, otherwise there
exists a countable $T=\{t_n: n\in\w\}\subset Y'$ which is closed
discrete in $Y'$. Note that $Z = T \cap Y$  is finite, otherwise $x_Z$
would be an accumulation point of $T$ in $Y'$.
So we may assume that $T\cap Y=\emptyset$. For every $n$ fix
$Z_n \in [Y]^\w$ such that $t_n=x_{Z_n}$ and set $S_n:=S_{Z_n}$.
Let $W_n\ni
t_n$ be a neighborhood of $t_n$ in $Y'$ such that
$W_n\cap W_m=\emptyset$ for all $n\neq m$. [In a $T_3$ space, every countable closed
discrete subspace extends to a disjoint open collection.]
\medskip

For every $y\in Y'\setminus T$ find a basic open neighborhood $U_y$ of $y$ in
$Y'$ such that $\overline{U_y}\cap T=\emptyset$. Given an
injective enumeration $\langle z^n_k:k\in\w\rangle$ of $S_n$ for every $n \in \w$,
let $f_{U_y} :\w \to \w$
be such that $\{z^n_k:k\geq f_{U_y}(n)\}\cap\overline{U_y}=\emptyset$.
Since the weight of $Y'$ is $< \mathfrak b$, there are $<\mathfrak b$ of $U_y$'s, and hence there exists $f : \w \to \w$ such that
$f_{U_y}<^* f$ for all $y\in Y' \setminus T$. Now
$Z_\w = \{z^n_{f(n)}:n\in\w\}\subset Y$ has no accumulation points
in $Y'$, because $Z_\w \cap W_n = \{z^n_{f(n)}\}$ for all $n\in\w$, and $Z_\w \cap U_y$
is finite for each $y \in Y'$. In other words,  $x_{Z_\w}$ does not exist.
This contradiction
implies that $Y'$ is countably compact.
\end{proof}

\medskip
For locally compact spaces,
we can replace the hypothesis in the statement of Lemma \ref{Z} with
 $\ob{Y}$ being hereditarily of Lindel\"of degree $< \frak p$. That
is, if $S \subset \ob{Y}$, and $\cal U$ is a relatively open cover
of $S$, then $\cal U$ has a subcover of cardinality $< \frak p$.

Together with regularity this yields that $\ob{Y}$ has character
$<\mathfrak p$, and hence for any $Z\in [Y]^\w$ and limit point
$x_Z\in X$ of $Z$ there is a convergent to $x_Z$ sequence
$S_Z\in[Z]^\w$. This is a well-known property of spaces with
character $<\mathfrak p$.

Regarding the second part of the proof, the existence of $\cal
U'\subset \cal U = \{U_y: y \in S\}$ of  cardinality $< \frak p$
covering  $S = Y'\sm T$, is enough to make the argument go through.
\bigskip

Now we can finish answering Question \ref{1.1}.

\begin{thm} \label{Tri1}
Assume the PID and $\min\{\mathfrak b,\mathfrak s\} > \A_1$.
Then every locally compact, $\w_1$-compact space
of weight $\A_1$ is
$\s$-countably compact. \end{thm}

\begin{proof} Again Lemmas \ref{a1} and \ref{b} and $\w_1$-compactness
give us alternative $(1^-)$ of Theorem \ref{tri1}. The rest is
clear from Lemma \ref{Z}. \end{proof}

\bigskip

Returning to Theorem \ref{Tri2}, its proof also gives:

\begin{thm} \label{dowker} Assume the PID. Then every locally
compact, $\w_1$-compact normal space of cardinality $< \b$
is countably paracompact.\end{thm}

\begin{proof}
 A normal space $X$ is countably paracompact if, and only if, for
each descending sequence of closed sets $\langle F_n\rangle_{n=1}^\infty$
with empty intersection, there is a sequence of open sets  $\langle U_n\rangle_{n=1}^\infty$
with empty intersection, with $F_n \subset U_n$ for all $n$.
If $X$ is a countable union of countably compact subsets $C_m$,
as in Theorem \ref{Tri2}, then in such
a sequence of closed sets $F_n$, we can only have $F_n \cap C_m \ne \nix$ for finitely
many $n$. [Otherwise, countable compactness of $C_m$ implies
$\bigcap_{n=1}^\infty C_m \cap F_n \ne \nix.$]  In any Tychonoff
space, every pseudocompact subspace, and hence every countably compact subspace,
has pseudocompact closure, and every normal, pseudocompact space is countably
compact; and  so the complements of the sets $\ob{C_m}$
form the desired sequence of open sets.
\end{proof}
\bigskip

The equivalence in the preceding proof is shown in \cite[21.3]{W} and
 is due to Dowker, who also showed its
equivalence with $X \times [0,1]$ being normal. In honor of his pioneering work,
normal spaces that are {\it not} countably paracompact are called ``Dowker spaces.''
 Theorem \ref{dowker} thus implies the consistency of
 there being no locally
compact, $\w_1$-compact Dowker spaces of cardinality $\A_1$.
Specialized
though this fact is, it is one of the few theorems as to what kinds of Dowker
spaces are unattainable in ZFC. Another interesting such result was obtained
 by Dow and Tall \cite{DT}:

\begin{thm} \label{FT}  If $MM(S)[S]$, then every locally compact, normal,
non-paracompact space of Lindel\"of number $\le \A_1$ includes a
perfect preimage of $\w_1$. \end{thm}

Alan Dow has improved ``perfect preimage'' to ``copy'' [private communication].
Combining this with  Theorem \ref{dowker}, we have:

\begin{cor}  If $MM(S)[S]$, then every locally compact
Dowker space of cardinality $\le \A_1$ includes both a
copy of $\w_1$ and an uncountable closed discrete subspace.
 \end{cor}

 This corollary may be vacuous; in fact, the following problem is open:
\medskip

\noindent{\bf Problem 3.} Is there a model of PID + $\frak c = \A_2$
in which there is a locally compact Dowker space of cardinality
$\aleph_1$?

\bigskip
The consistency of PID was shown using forcing from a ground model with a
supercompact cardinal. There are versions for spaces of weight
$\aleph_1$, hence all locally compact spaces of cardinality $\aleph_1$, which
require only the consistency of ZFC. One restricted version of the PID axiom
is designated $(*)$ in \cite{AT}, and is adequate
 for Theorem \ref{dowker}. But it is still an open problem whether the main
results of our next section are ZFC-equiconsistent.

\section{\label{T5} When hereditary normality implies $\s$-countable
compactness}

For the main  theorem of this section, we recall the following concepts:

\begin{defn}\label{scwH} \rm  Given a subset $D$ of a set $X$,
an {\it expansion} of $D$ is a family $\{U_d : d \in D\}$ of subsets of $X$
such that $U_d  \cap D = \{d\}$ for all $d \in D$.
 A space $X$ is {\it [strongly] collectionwise
Hausdorff }  (abbreviated {\it [s]cwH})
if every closed discrete
subspace has an expansion to a disjoint [{\it resp.} discrete] collection
of open sets.
\smallskip

The properties of {\it $\w_1$-[s]cwH} only require taking care of those
$D$ that are of cardinality $\le \w_1$.
 \end{defn}

A well-known, almost trivial  fact is that every normal, cwH space
is scwH:  if $D$ and $\{U_d : d \in D\}$ are as in \ref{scwH}, let
$V$ be an open set containing $D$ whose closure is in
$\bigcup \{U_d : d \in D\}$; then $\{U_d \cap V : d \in D\}$
is a discrete open expansion of $D$.
\bigskip

As is well known \cite[2.1.7]{E}, a space is hereditarily
 normal if (and only if) every open subspace is normal.
A similar statement holds for the cwH property:

\begin{thm} \label{openh}The following are equivalent for a space $X$.
\smallskip

\noindent
(1) $X$ is hereditarily cwH.
 \smallskip

\noindent
(2) Every open subspace of $X$ is cwH.
\smallskip

\noindent
(3) Every discrete subspace of $X$ expands to a disjoint collection of open sets.
\end{thm}

\begin{proof} To show {\it (2)} implies {\it (3)}, let $D$ be a discrete subspace of $X$. Then $D' = \ol{D} \sm D$ is a closed subset of $X$,
and $D$ is a closed discrete subspace of the open subspace $Y = X \sm (\ol{D} \sm D) $ of $X$. Clearly, any expansion of $D$ to disjoint open sets in $Y$
is also a disjoint open expansion in $X$.

\medskip
\noindent
To show  {\it (3)} implies {\it (1)}, let $S$ be a subspace of $X$, and let $D$ be a relatively closed, discrete subspace of $S$. Then $D$
is also a discrete subspace of $X$, and the trace on $S$ of a disjoint open expansion of $D$ in $X$ is a disjoint relatively open expansion of $D$ in $S$.
\end{proof}
\medskip

The analogous theorems for  scwH,  $\w_1$-scwH and $\w_1$-cwH also hold, and the proofs are similar.
We will be using the one for  $\w_1$-cwH in the main theorem of this section.

\bigskip

The following class of normal spaces plays a big role in this section and in the following one.

\begin{defn} \label{monO} \rm A space $X$ is
{\it  monotonically normal} if there is
an operator $G(\_,\_)$ assigning to each ordered pair
$\langle F_0, F_1\rangle$ of disjoint closed subsets an open set
$G(F_0, F_1)$ such that \newline
(a) $F_0 \subset G(F_0, F_1)$\newline
(b) If $F_0 \subset F_0'$ and $F_1' \subset F_1$ then $G(F_0, F_1) \subset G(F_0', F_1')$\newline
(c) $G(F_0, F_1) \cap G(F_1, F_0) = \nix.$
\end{defn}

Monotone normality is a hereditary property; that is, every subspace inherits
the property. This is not so apparent from this definition, but it follows almost
immediately from the following criterion, due to Borges \cite{B}.

\begin{thm} \label{bor} A space is monotonically normal if, and only if,
there is an assignment of an open neighborhood  $h(x, U) =: U_x$
containing $x$ to
each pair $(x, U)$
such that $U$ is an open neighborhood of $x$,
and such that, if $U_x \cap V_y \ne \nix,$
then either $x \in V$ or $y \in U$.\end{thm}

Every monotonically normal space is (hereditarily) countably paracompact \cite[Theorem 2.3]{R0}
and (hereditarily) scwH: the Borges criterion easily give cwH, and normality does the rest.
\bigskip

The main theorem of this section, Theorem \ref{Tri3}, also involves the following
concepts.

\begin{defn} \rm An 
{\it $\w$-bounded space} is one in which
every countable subset 
has compact closure. A space is {\it $\s$-$\w$-bounded}
if it is the union of countably many $\w$-bounded
subspaces. \end{defn}

Clearly, every  $\w$-bounded space is  countably compact.
Theorem \ref{Tri3} below makes use of the following axiom, which
is shown in \cite{EN} to follow from PID and whose numbering
is aligned with that of Theorem \ref{tri1}:

\begin{ax} \label{tri2} Every
locally compact  space has either:
\smallskip

\noindent
$(1)$ A countable collection of $\w$-bounded subspaces
whose union is the whole space or
\smallskip

\noindent
$(2)$
 An uncountable closed discrete subspace or
\smallskip

\noindent
$(3)$ A countable subset with non-Lindel\"of closure.
\end{ax}

Of course, $(1)$ and $(2)$ are mutually exclusive, but each is compatible
with $(3)$.
\bigskip

Part $(iii)$ of our main theorem is the promised strengthening
of Theorem \ref{mm}.

\begin{thm} \label{Tri3} Let $X$ be a locally compact, $\w_1$-compact
space. If either
\smallskip

\noindent
$(i)$ $X$ is monotonically normal and the P-Ideal Dichotomy (PID) axiom holds, or
\smallskip

\noindent
$(ii)$ $X$ is hereditarily $\w_1$-scwH, and either
PFA or PFA(S)[S] hold, or
\smallskip

\noindent
$(iii)$ $X$ is hereditarily normal, and MM(S)[S] holds,
\smallskip

\noindent THEN $X$ is $\s$-$\w$-bounded, and is either Lindel\"of
(and hence $\s$-compact) or contains a copy of $\w_1$.
\end{thm}

\begin{proof}
The $\s$-$\w$-bounded property in case $(i)$ follows from the fact that,
in a monotonically normal space, every countable subset has
Lindel\"of closure \cite{O}, and from Axiom \ref{tri2}.

\medskip

The PID, and hence Axiom \ref{tri2}, holds in any model of $PFA$ or
$PFA(S)[S]$,
so $\s$-$\w$-boundedness in case $(ii)$ follows from the following two facts.
First, every point (and hence every countable subset) of a locally
compact space has an open Lindel\"of neighborhood (see \cite[Lemma 1.7]{N4}).
Second, in these models, every open Lindel\"of subset of a
hereditarily $\w_1$-scwH space
has Lindel\"of closure (see \cite{N4} and \cite{N6}, respectively).
So again, $(3)$ in Axiom \ref{tri2} fails outright, and $\s$-$\w$-boundedness
 for $(ii)$ follows from (1)
of Axiom \ref{tri2}.
\medskip

As for $(iii)$, Dow and Tall \cite{DT} have shown that $MM(S)[S]$ implies
that every normal, locally compact space is $\w_1$-cwH.
This implies that every hereditarily normal, locally compact space is hereditarily
$\w_1$-scwH by the following reasoning. First, every open subspace of a locally compact space is locally compact. Therefore, if an open subspace is normal, then by the Dow-Tall theorem, it is  $\w_1$-cwH. Then, by a trivial variation on Theorem \ref{openh}, the space is hereditarily  $\w_1$-cwH. Finally, by the comment following Definition \ref{scwH}, it is hereditarily  $\w_1$-scwH.  Now we can continue
as for $(ii)$.
\medskip

That $X$ is either Lindel\"of or contains a copy of $\w_1$ in case
$(i)$, is immediate from the following fact, whose proof is deferred
to the end of this section:
\smallskip

{\it Every locally compact, monotonically normal space
is either paracompact or contains a closed copy of a regular
uncountable cardinal.}
\medskip

The Lindel\"of alternative for case $(i)$ uses the fact  \cite[5.2.17]{E} that every locally compact,
paracompact space is the union of a disjoint family of
(closed and) open Lindel\"of subspaces. Now $\w_1$-compactness
makes the family countable, and so $X$ is Lindel\"of.
\medskip

\medskip
The same either/or alternative for $(ii)$ and $(iii)$
uses the ZFC theorem that
any $\w$-bounded, locally compact space is either Lindel\"of or  contains a
perfect preimage of $\w_1$,
along with the reduction of character
in Lemma 1.2 of \cite{N4}, which uses $MA(\w_1)$, which in turn is
implied by the PFA. Moreover,
the proof of this lemma carries over to $PFA(S)[S]$. This proof shows
that in a locally compact, hereditarily $\w_1$-scwH space,
every open Lindel\"of subset has Lindel\"of closure and hereditarily
Lindel\"of (hence first countable) boundary.
\smallskip

Now, given a perfect surjective map $\varphi: W \to \w_1$, let $Y$
be the union of the boundaries of the subsets $\varphi^\from [0,
\a)$ where $\a$ is a limit ordinal. Letting $\Lambda $ be the set of
all limit countable ordinals, note that $\psi:Y\to\Lambda$ such that
$\psi^\from(\{\alpha\})=\partial(\varphi^\from[0,\alpha))\subset\varphi^\from(\{\alpha\})$
(here the boundary is computed of course in $W$) is a perfect map
with first-countable fibers, the compactness of $\varphi^\from [0,
\a]$ guaranteeing that
$\partial(\varphi^\from[0,\alpha))\neq\emptyset$ and thus the
surjectivity of $\psi$. Since perfect preimages of locally compact
first-countable spaces under maps with first-countable fibers are
also  first-countable, we can
 apply to $Y$
the theorem \cite{DT1} that $PFA(S)[S]$ implies that every first
countable perfect preimage of $\w_1\cong\Lambda$ contains a copy of
$\w_1$. That the $PFA$ also implies this is a well-known theorem of
Balogh, shown in \cite{D}.
\end{proof}

\begin{rem}\rm In the conclusion of Theorem \ref{Tri3}, one cannot expect the copy of $\w_1$
to be closed. The ordinal $\w_2$ is monotonically normal (as is any
linearly orderable space) and $\w$-bounded, but is not Lindel\"of,
and every copy of $\w_1$ inside it is bounded, hence not closed.
\end{rem}

The following theorem can be derived from Theorem \ref{Tri3} $(ii)$ in the same
way that Theorem \ref{dowker} is derived from Theorem \ref{Tri2}.

\begin{thm} \label{Dowker} Let $X$ be a locally compact, $\w_1$-compact, normal,
hereditarily $\w_1$-scwH
space. If either PFA or PFA(S)[S] holds, then $X$ is countably paracompact.
\end{thm}

To put it
positively, every ZFC example of a locally compact, hereditarily $\w_1$-scwH
Dowker space must contain an uncountable closed discrete subspace.
However, to be absolutely certain of this,
we need a negative answer to the second part of the following question:

\medskip

\noindent
{\bf Problem 4.} {\it Do we need the large cardinal strength of (at least) PID
 to obtain any or all of the conclusions of Theorem \ref{Tri3}
or Theorem \ref{Dowker}?}
\medskip

Where Theorem \ref{Dowker} is concerned, the other applications of $PFA$, etc.
are taken care of by axioms $MA(\w_1)$ and $MA(S)[S]$ \cite{N4}, \cite{N7},
both of which are equiconsistent with ZFC.  However, we may
need something more for the conclusion about the copy of $\w_1$,
or for
the theorem \cite{DT} that $MM(S)[S]$ implies that every locally
compact normal space is $\w_1$-cwH, or for the following corollary of
this fact and of Theorem \ref{Dowker}:

\begin{cor} If $MM(S)[S]$, then every locally compact,
hereditarily normal Dowker space must contain an uncountable closed
discrete subspace. \qquad $\square$ \end{cor}


We now complete the proof of Theorem \ref{Tri3}. First, we  recall a
powerful theorem of Balogh and Rudin \cite{BR}.

\def \E{\cal E}

\begin{thm} \label{br} Let $X$ be a monotonically normal space.
    The following are equivalent.
\medskip

$(1) \ \ X$ is paracompact.
\medskip

$(2) \ \ X$ does not have a closed subspace homeomorphic to a
stationary subset of a regular uncountable cardinal.\end{thm}
\medskip

\begin{lemma} \label{tail} Let $\d$ be an ordinal of uncountable
cofinality, and let $E$ be a locally compact stationary subset of
$\d$. There is a tail (final segment) $T$ of $E$ which is a closed
(hence club) subset of $\d$.\end{lemma}

\begin{proof}
Each $x\in E$ has a compact  open neighborhood $H_x$ of the form
$[\be_x, x]  \cap E$ whose least element is isolated in $E$. Then
the pressing down lemma  gives us a single $\be$ which works for a cofinal subset $S$
of $E$, and so $\bigcup\{[\be, x] \cap E : x \in S\}$ is a tail of
$E$ and a club in $\d$. \end{proof}

Lemma \ref{tail} gives the strengthening of Theorem \ref{br} for
locally compact spaces that was used in the proof of Theorem
\ref{Tri3}:

\begin{cor} \label{rb} Let $X$ be a monotonically normal, locally compact space.
The following are equivalent.

\medskip

        $(i) \ \ X$ is paracompact.
\medskip

$(ii) \ \ X$ does not contain  a closed copy of
a regular uncountable cardinal. \end{cor}

\begin{proof} We show that $(ii)$ is equivalent to (2) of Theorem \ref{br}.  Every ordinal is a
stationary subset of itself; hence if $(2)$ is true then so is $(ii)$. Conversely, every closed 
subset of $X$ is locally compact. Therefore, if $E$ is a closed copy
in $X$ of a  stationary subset of the regular uncountable cardinal
$\d$, then  the $T$ obtained in Lemma \ref{tail} is itself a copy of
$\d$. \end{proof}

\section{\label{mono} Monotonically normal examples under $\clubsuit$ for Question 1.1}

This section features  a remarkably simple construction of consistent examples of locally compact,
 $\w_1$-compact, monotonically
normal spaces of cardinality $\aleph_1$ that
are not $\s$-countably compact.

\bigskip

Our construction gives the ordinal $\w_1$ a locally compact topology $\tau$ that is
finer than the usual (order) topology, in which the $\tau$-relative topology
on the subspace $\La$ of limit ordinals is its usual order topology. As for the successor
ordinals $< \l \in \La$, they will include closed discrete subspaces with supremum $\l,$
and sequences converging to  $\l$. The axiom $\clubsuit$ is used to ensure that
these two behaviors result in both $\w_1$-compactness and failure of $\s$-countable compactness.
\bigskip

The following concept is part of the most widely used statement of $\clubsuit$.

\begin{defn} \rm Given a countable  limit ordinal $\a$, a {\it ladder at $\a$}
is an infinite subset $L_\alpha$ of $\alpha$ such that $|L_\alpha\cap\beta|<\w$ for all
$\beta\in\alpha$.  A {\it ladder system on $\w_1$}
is a family
$$
\cal L = \{L_\a : \a \in \La\} \text{ where each $L_\a$ is a ladder at $\a$}.
$$
\end{defn}


\begin{ax} \label{clubs} Axiom  $\clubsuit$
\rm states that there is a ladder system $\L$ on $\w_1$ such that,
for any uncountable
subset $S$ of $\w_1$, there is $L_\a \in \L$
such that $L_\a \subset S$.
\end{ax}

We now define a general family $\frak S$ of spaces similar to the two
families in \cite{JN}.

\begin{nota} \label{Nota}
 Let $\frak S$
designate the set of topologies $\tau$ on $\w_1$
in which, to each point  $\a$ there are associated
$B(\a) \subset [0, \a]$ and
$B(\a, \xi) = B(\a) \cap (\xi, \a]$ for each $\xi < \a$,
such that:
\smallskip

 \noindent $(1)$ $ \{B(\a, \xi) : \xi < \a\}$ is a base for the neighborhoods of
$\a$ [we allow $\xi = -1$ in case $\a = 0$].

 \smallskip

\noindent $(2)$ If $\a\in \La $  and $\be > \a$, then there exists
$\xi < \a$ such that $B(\a, \xi) =
B(\be) \cap (\xi, \a]$. [In particular, $\a \in B(\be)$.]
\smallskip

\noindent $(3)$ There is a ladder system $\L = \{L_\a : \a \in \La$\},
such that if
$M_\a = \{\xi+1: \xi \in L_\a\}$,
 then $\a = sup [M_\a \cap B(\a)] = sup [M_\a  \setminus B(\a)]$.
\end{nota}

\bigskip

One  role of $M_\a$ is to simplify the construction of spaces in $\frak S$ by
not having to deal with splitting $L_\a$ between limit and isolated ordinals.
It has the novel effect of translating the ladders in $\clubsuit$ to $\w_1 \sm \La$, but
preserving the action of $\clubsuit$ to give $(\w_1, \tau)$ the desired properties.
\bigskip

We do not need $\clubsuit$ for the following lemma:

\begin{lemma} Every space in $\frak S$ is locally compact, and the $\tau$-relative topology
on $\La$ is the usual order topology. \end{lemma}

\begin{proof} The second conclusion is immediate from (2).  Obviously, $B(0) = \{0\}$ and
 $B(\xi + 1, \xi)$ are singletons for all
successor ordinals $\xi+1$.   If $\be \in \La$,
 then, since $B(\be)$ is
countable, it is enough to show that $B(\be)$ is countably compact to show that it is compact.
\smallskip

 \medskip

Let $A$ be an infinite subset of
$B(\be)$. Then $A$ contains a strictly ascending sequence $\sigma$ of ordinals. Let
$sup(ran (\sigma)) = \a$. If $\a = \be$ then $\sigma \to \a$ by (1), while if
$\a < \be$, then (1) and (2) have the same effect, implying that $\a$ is a limit point of
$A$ in $B(\be)$. \end{proof}

Next, we show how to construct spaces in $\frak S$ by induction.

\begin{exs} \rm Let $\tau$ be a topology defined using $\L$,
and using the base
$B(\a)$ and $B(\a, \xi)$ defined by recursion as follows:
\smallskip

Let $B(0) = \{0\}$ and, if $\a = \xi +1$, let $B(\a) = B(\xi) \cup \{\a\}.$
Given $L_\a \in \L$, let $M_\a = \{\xi+1: \xi \in L_\a\}$
be  listed  in strictly ascending order as $\{\a_n : n \in \w\}$.
If  $\a = \nu + \w$ where $\nu$ either is $0$ or a limit
ordinal,  let $B(\a) = B(\nu)
\cup (\nu, \a] \setminus \{\a_{2n} : n \in \w\}$.
\smallskip

\medskip

If $\a \in \La$ is not of the form $\nu + \w$,  let
$S(\a, 0) = B(\a_0)$  and, for $n > 0$, let
$
S(\a, n) = B(\a_n, \a_{n-1})
$
and let
$$
B(\a) = \bigcup_{n=0}^\infty \left( S(\a, n) \cup \{\a_{2n+1} : n \in \w\} \right)
\setminus \{\a_{2n} : n \in \w\}.
$$

\noindent
{\it Claim} $(\w_1, \tau) \in \frak S.$
\smallskip

\noindent
$\vdash$ \quad Assuming (2) in Notation \ref{Nota}, we first show (1).
If $\a \in B(\be, \eta) \cap B(\g, \nu)$ and $\a > \nu,\ \a \notin \La$
 then obviously $B(\a, \a-1) = \{\a\}$ works, i.e.,
 $\alpha\in B(\a, \a-1)\subset B(\be, \eta) \cap B(\g, \nu)$.
Otherwise,  we can find a
basic open neighborhood of $\a$ inside the intersection by using (2). Just
choose $\xi$ greater than $\eta$ and $\nu$ and  large enough so that
$B(\a, \xi) = B(\be) \cap (\xi, \a]$ and also $B(\a, \xi) = B(\g) \cap (\xi, \a]$.
\medskip

Next we show (2)  by induction. Suppose it is false,
and that $\be$ is the first ordinal for which
there is a limit ordinal $\a < \be$
where it fails. Clearly $\be \in \La$. Only successor ordinals are
in $[0, \be] \setminus B(\be)$, so $\a \in B(\be)$.
Let $\a \in S(\be, n) = B(\g, \nu)$.
By minimality of $\be$, there exists $\mu \ge \nu$ such that
$B(\a, \mu) =  B(\g, \nu) \cap (\mu, \a].$
However, $S(\be, n)$ only contains finitely many members of
$L_\be$; once $\xi$ is above all the members of $L_\be$
that are below $\a$, we must have
$B(\be) \cap (\xi, \a] =  B(\g, \nu)  \cap (\xi, \a]$,
and if $\xi \ge \mu$ then this equals $B(\a, \xi)$, a contradiction.
\medskip

Finally, (3) is true by construction since $\a$ is the supremum of the points
of $M_\a$ that have odd subscripts and also of the ones
that have even subscripts. \quad $\dashv$
 \end{exs}
\bigskip

Now we come to the main theorem of this section.

\begin{thm} \label{suit} Suppose that $\tau\in\frak G$ and $\mathcal L$ is a
ladder system such that item $(3)$ of Notation~\ref{Nota}
holds for $\tau$ and $\mathcal L$.
If in addition $ \L$ witnesses $\clubsuit$, then
$(\w_1, \tau)$ is monotonically normal and $\w_1$-compact but not
$\s$-countably compact.\end{thm}

\begin{proof} Recall the Borges criterion  for monotone
normality, Theorem \ref{bor}:
\medskip

{\it There is an assignment of an open neighborhood  $h(x, U) \
=: U_x$ containing $x$ to
each pair $(x, U)$
such that $U$ is an open neighborhood of $x$,
and such that, if $U_x \cap V_y \ne \nix,$
then either $x \in V$ or $y \in U$.}
\medskip

The choice of $h(z, U) = \{z\}$ for all isolated points $z$
works for the case where either
$x$ or $y$ is not a limit ordinal. So it is enough to take care of the
case where  $x$ and $y$ are both
limit ordinals and $x < y$. Given $\a \in U,$ let
$U_\a = B(\a, \xi)$ for some $\xi$ such that $B(\a, \xi) \subset U$.
It follows from (2) that
$x \in U_y$ whenever $U_x \cap U_y \ne \nix.$
\bigskip

It also follows from (2) that the $\tau$-relative topology
on $\La$ is its usual order topology. So, to show $\w_1$-compactness
it is enough to show that every uncountable set $S$ of successor
ordinals has an ascending sequence $\langle s_n : n \in \w \rangle$
that $\tau$-converges to its supremum.
\smallskip

Let $R = \{\xi : \xi +1 \in S\}$.
Since $\L$ witnesses $\clubsuit$, there exists $\a \in R$ such that
$L_\a \subset R$. Then $M_\a \subset S$ and the odd-numbered
elements of $M_\a$ converge to $\a$.

\bigskip

Finally, to show that $(\w_1, \tau)$ is not $\s$-countably compact,
let $\w_1 = \bigcup_{n=1}^\infty A_n$. It is clearly enough to show that any
$A_n$ that contains uncountably many successor ordinals is
not countably compact, since there is at least one such $A_n$.
Let one of these be $S$, and let $R$ and $M_\a$ be as before.
Then the even-numbered members of $M_\a$ are an infinite
$\tau$-closed discrete subset of $\w_1$. \end{proof}
\bigskip

The way that the Borges criterion was shown for the spaces in Theorem \ref{suit}  makes it clear
that they witness the following property.

\begin{defn} \rm A space $X$ is {\it utterly ultranormal} if there is a system of basic neighborhoods
$\cal B_x$ for each $x\in X$ satisfying
\medskip

$(\ddagger)$ \qquad  If $B_x \in \cal B_x$ and $B_y \in \cal B_y$  and $B_x \cap B_y \ne \nix,$
then either $x \in {B_y}$ or $y  \in B_x$. \end{defn}
\medskip

For the more general concept of {\it utterly normal} see \cite{CJN}.
Every utterly normal space is monotonically normal, but the converse
is a  long-standing unsolved problem.


\section{\label{min} The minimum cardinality theme}

We return now to Problems 1 and 2, repeated here for convenience:
\medskip

\noindent
{\bf Problem 1.} {\it What is the minimum cardinality  of
a locally compact, $\w_1$-compact space that is not
$\s$-countably compact? one that is normal?}

\medskip

\noindent
{\bf Problem 2.} {\it Is there a ZFC example of a  locally compact,
$\w_1$-compact space of cardinality $\A_2$ that is not
$\s$-countably compact? one that is normal?}

\medskip

Of course, if $\b = \aleph_1$, then Example \ref{minb} works for Problem 1, and  thus the topological direct sum thereof
with   a compact space of cardinality $\A_2$ works for Problem 2. Similarly, if  $\b = \aleph_2$, then the answer to Problem 2 is Yes, and hence the answer to Problem 1 is $\leq\A_2$.

\bigskip

Theorem~\ref{Tri1} gave us extra set-theoretic conditions under which the answer to Problem 1 is
$\geq \aleph_2$.  Let us note that
the proof of Theorem~\ref{Tri1} for spaces of size $\aleph_1$
 only used the weakening of PID to $P$-ideals on sets
 of size also $\aleph_1$,
 designated $PID_{\aleph_1}$.
Since $PID_{\aleph_2}$ implies $\mathfrak b\leq \aleph_2$ by a result of Todor\v{c}evi\'c,
 it is impossible  to make the same technique  as in the proof of Theorem~\ref{Tri1} work to give us a model where the answer to Problem 1 would be
 $\aleph_3$, as
we would need the inconsistent assumption $PID_{\aleph_2} +
\mathfrak b > \aleph_2$. For the sake of completeness we prove here
the above-mentioned inconsistency. First, some notation and a
definition.

\medskip

Given a relation $R$ on $\omega$ and $x,y\in\omega^\omega$, we
denote the set $\{n\in\omega:x(n)\:R\: y(n)\}$ by $[xRy]$.

\begin{defn} \rm Let $\kappa,\lambda$ be regular cardinals. A {\it
$(\kappa,\lambda)$-pregap} $\langle \{f_\alpha\}_{\alpha<\kappa},
\{g_\beta\}_{\beta<\lambda}\rangle$
 is a pair of
transfinite sequences $\langle f_\alpha:\alpha<\kappa\rangle$ and
$\langle g_\beta:\beta<\lambda\rangle$ of nondecreasing sequences
$f_\alpha, g_\beta$ of natural numbers such that
$f_{\alpha_1}\leq^\ast f_{\alpha_2}\leq^\ast g_{\beta_2}\leq^\ast
g_{\beta_1}$ for all $\alpha_1\leq\alpha_2<\kappa$ and
$\beta_1\leq\beta_2<\lambda$. As usual, $f\leq^\ast g$ means that
the set $[f>g]$  is finite. A $(\kappa,\lambda)$-pregap is called a
{\it $(\kappa,\lambda)$-gap}, if there is no $h\in\omega^\omega$
such that $f_\alpha\leq^\ast h\leq^\ast g_\beta$ for all
$\alpha,\beta$. \end{defn}

 Lemma~\ref{gap1} is a slight improvement of
\cite[Lemma~1.12]{MZ}, which in its
 turn  is modelled on the second paragraph of
\cite{Todorcevic2000}.

\begin{lemma} \label{gap1}
Suppose that there exists a $(\kappa,\lambda)$-gap such that
$\kappa,\lambda$ are regular uncountable.
\begin{enumerate}
\item If $\kappa>\aleph_1$ then
 $\mathrm{PID}_{\lambda}$ fails.
\item If $\lambda>\aleph_1$ then
 $\mathrm{PID}_{\kappa}$ fails.
\end{enumerate}
 \end{lemma}
\begin{proof}

The proofs of the two parts of the lemma are essentially the same,
but we present both of them for the sake of completeness.
\medskip

1. Assume that $\kappa>\aleph_1$ but  $\mathrm{PID}_{\lambda}$
holds. Fix a $(\kappa,\lambda)$-gap
$\langle\{f_\alpha\}_{\alpha<\kappa},\{g_\beta\}_{\beta<\lambda}\rangle$
and set
 $$\mathcal{I}=\{A\in[\lambda]^{\aleph_0} \ :\ \exists\alpha \in\kappa\:\forall n\in\omega\:
(|\{\beta\in A: [f_\alpha>g_\beta]\subset n \}|<\aleph_0)\}. $$ Note
that if $\alpha$ witnesses $A\in\mathcal{I}$, then any
$\alpha'>\alpha$ has also this property because $f_{\alpha}\leq^*
f_{\alpha'}$.
 We claim that $\mathcal{I}$ is a $P$-ideal. Indeed,
let $\{A_i:i\in\omega\}$ be a sequence of   elements of
$\mathcal{I}$, $\alpha_i$ be a witness for $A_i\in\mathcal{I}$, and
$\alpha=\sup\{\alpha_i:i\in\omega\}$. Then $\alpha$ witnesses
$A_i\in\mathcal{I}$ for all $i\in\omega$.  Let $B_i =\{\beta\in A_i:
[f_\alpha>g_{\beta}] \subset i \}$. Then $B_i$ is a finite subset of
$A_i$ by the definition of $\mathcal{I}$. Set
$A=\bigcup_{i\in\omega}A_i\setminus B_i$ and fix $n\in\omega$. If
$\beta\in A_i$ is such that $[f_\alpha>g_\beta]\subset n$ and $i\geq
n$, then $\beta\in B_i$. Therefore
$$\{\beta\in A: [f_\alpha>g_\beta]\subset n \}\subset\{\beta\in \bigcup_{i<n} A_i: [f_\alpha >g_\beta]\subset n\},$$
and the latter set is finite by the definition of $\mathcal{I}$ and
our choice of $\alpha$.
\smallskip

Applying  $\mathrm{PID}_{\lambda}$ to $\mathcal I$ we conclude that
one of the following alternatives is true:
\smallskip

1a. There exists  $S\in [\lambda]^{\aleph_1}$ such that
$[S]^{\aleph_0}\subset\mathcal{I}$. Passing to an uncountable subset
of $S$ if necessary, we can assume that
$S=\{\beta_\xi:\xi<\omega_1\}$ and $\beta_\xi<\beta_\eta$ for any
$\xi<\eta<\omega_1$. For every $\xi$ we denote by $S_\xi$ the set
$\{\beta_\zeta:\zeta<\xi\}$.

By the definition of $\mathcal{I}$ for every $\xi$ there exists
$\alpha_\xi\in\kappa$ witnessing for $S_\xi\in \mathcal{I}$.  Let
$\alpha=\sup\{\alpha_\xi:\xi<\omega_1\}$. There exists $n\in\omega$
such that the set $C= \{\xi<\omega_1: [f_{\alpha}>
g_{\beta_\xi}]\subset n\}$ is uncountable. Let $\xi_0$ be the
$\omega$-th element of $C$. Then $\alpha\geq \alpha_{\xi_0}$ and for
all $\xi\in C\cap\xi_0$ we have $ [f_{\alpha}> g_{\beta_\xi}]\subset
n$. On the other hand, $S_{\xi_0}\in\mathcal{I}$, and hence there
should be only finitely many such $\xi\in\xi_0$, a contradiction.
\smallskip

1b. $\lambda=\bigcup_{m\in\omega}S_m$ such that $S_m$ is orthogonal
to $\mathcal{I}$ for all $m\in\omega$. This  means that for every
$m\in\omega$ and $\alpha\in\kappa$  there exists
$n_{m,\alpha}\in\omega$ such that $[f_{\alpha}>g_\beta]\subset
n_{m,\alpha}$ for all $\beta\in S_m$. (If there is no such
$n=n_{m,\alpha}$, construct a sequence
$\langle\beta_n:n\in\omega\rangle\in S_m^\omega$ such that
$[f_{\alpha}>g_{\beta_n}]\not\subset n$ and note that $\alpha$
witnesses $\{\beta_n:n\in\omega\}\in\mathcal{I}$, which is
impossible because $S_m$ is orthogonal to $\mathcal{I}$). Since
$\lambda $ is regular uncountable,  there exists $m\in\omega$ such
that $S_m$ is cofinal in $\lambda$. Let $n\in\omega$ be such that
the set $J=\{\alpha\in\kappa:n_{m,\alpha}=n\}$ is cofinal in
$\kappa$. For every $k$ let $h(k)=\max\{f_{\alpha}(k):\alpha\in
J\}$. Note that $h(k)\in\w$  for $k\geq n_{m,k}$
because $f_{\alpha}(k)\leq g_{\beta}(k) $ for all $\alpha\in J$,
$\beta\in S_m$, and $k\geq n_{m,\alpha}$. For $k<n_{m,\alpha}$ the
fact that $h(k)\in\w$ follows from  $f_\alpha$'s being
non-decreasing.
 From the above it follows that $[g_\beta < h]\subset n$ for
all $\beta\in S_m$, and hence $h$ contradicts the fact that
 $\langle\{f_\alpha\}_{\alpha\in J},\{g_\beta\}_{\beta\in S_m}\rangle$
is a gap.
\medskip

2. Assume that $\lambda>\aleph_1$ but $\mathrm{PID}_{\kappa}$ holds.
Fix a $(\kappa,\lambda)$-gap
$\langle\{f_\alpha\}_{\alpha<\kappa},\{g_\beta\}_{\beta<\lambda}\rangle$
and set
 $$\mathcal{I}=\{A\in[\kappa]^{\aleph_0} \ :\ \exists\beta \in\lambda\:\forall n\in\omega\:
(|\{\alpha\in A: [f_\alpha>g_\beta]\subset n \}|<\aleph_0)\}. $$
Note that if $\beta$ witnesses $A\in\mathcal{I}$, then any
$\beta'>\beta$ has also this property because $f_{\beta'}\leq^*
f_{\beta}$.
 We claim that $\mathcal{I}$ is a $P$-ideal. Indeed,
let $\{A_i:i\in\omega\}$ be a sequence of   elements of
$\mathcal{I}$, $\beta_i$ be a witness for $A_i\in\mathcal{I}$, and
$\beta=\sup\{\beta_i:i\in\omega\}$. Then $\beta$ witnesses
$A_i\in\mathcal{I}$ for all $i\in\omega$.  Let $B_i =\{\alpha\in
A_i: [f_\alpha>g_{\beta}] \subset i \}$. Then $B_i$ is a finite
subset of $A_i$ by the definition of $\mathcal{I}$. Set
$A=\bigcup_{i\in\omega}A_i\setminus B_i$ and fix $n\in\omega$. If
$\alpha\in A_i$ is such that $[f_\alpha>g_\beta]\subset n$ and
$i\geq n$, then $\alpha\in B_i$. Therefore
$$\{\alpha\in A: [f_\alpha>g_\beta]\subset n \}\subset\{\alpha\in \bigcup_{i<n}
A_i: [f_\alpha >g_\beta]\subset n\},$$ and the latter set is finite
by the definition of $\mathcal{I}$ and our choice of $\beta$.

Applying  $\mathrm{PID}_{\kappa}$ to $\mathcal I$ we conclude that
one of the following alternatives is true:
\smallskip

2a. There exists  $S\in [\kappa]^{\aleph_1}$ such that
$[S]^{\aleph_0}\subset\mathcal{I}$. Passing to an uncountable subset
of $S$ if necessary, we can assume that
$S=\{\alpha_\xi:\xi<\omega_1\}$ and $\alpha_\xi<\alpha_\eta$ for any
$\xi<\eta<\omega_1$. For every $\xi$ we denote by $S_\xi$ the set
$\{\alpha_\zeta:\zeta<\xi\}$.

By the definition of $\mathcal{I}$ for every $\xi$ there exists
$\beta_\xi\in\lambda$ witnessing for $S_\xi\in \mathcal{I}$.  Let
$\beta=\sup\{\beta_\xi:\xi<\omega_1\}$ and note that it witnesses
$S_\xi\in\mathcal{I}$ for all $\xi<\omega_1$. On the other hand,
there exists $n\in\omega$ such that the set $C= \{\xi<\omega_1:
[f_{\alpha_\xi}> g_{\beta}]\subset n\}$ is uncountable. Let $\xi_0$
be the $\omega$-th element of $C$. Then $\beta$ fails to witness
$S_{\xi_0}\in\mathcal{I}$, a contradiction.
\smallskip

2b. $\kappa=\bigcup_{m\in\omega}S_m$ such that $S_m$ is orthogonal
to $\mathcal{I}$ for all $m\in\omega$. This  means that for every
$m\in\omega$ and $\beta\in\lambda$  there exists
$n_{m,\beta}\in\omega$ such that $[f_{\alpha}>g_\beta]\subset
n_{m,\beta}$ for all $\alpha\in S_m$. (If there is no such
$n=n_{m,\beta}$, construct a sequence
$\langle\alpha_n:n\in\omega\rangle\in S_m^\omega$ such that
$[f_{\alpha_n}>g_{\beta}]\not\subset n$ and note that $\beta$
witnesses $\{\alpha_n:n\in\omega\}\in\mathcal{I}$, which is
impossible because $S_m$ is orthogonal to $\mathcal{I}$). Since
$\kappa$ is regular uncountable, there exists $m\in\omega$ such that
$S_m$ is cofinal in $\kappa$. Let $n\in\omega$ be such that the set
$J=\{\beta\in\lambda:n_{m,\beta}=n\}$ is cofinal in $\lambda$. For
every $k$ let $h(k)=\min\{g_{\beta}(k):\beta\in J\}$. From the above
it follows that $[f_\alpha > h]\subset n$ for all $\alpha\in S_m$,
and hence $h$ contradicts the fact that
 $\langle\{f_\alpha\}_{\alpha \in S_m},\{g_\beta\}_{\beta \in J}\rangle$
is a gap.

\end{proof}

This proof also shows that the following weakening of  $\mathrm{PID}_\lambda$ fails in (1): 

\begin{ax} $\mathrm{PID}^w_\lambda$ \rm states that for each P-ideal $\cal I$ on a set $X$ of cardinality $|X|\geq\lambda$, either there is an uncountable set $A\subset X$ with $[A]^\w \subset \cal I$  or a set $B$ of cardinality $\lambda$ such that $B \cap I$ is finite
 for all $I \in \cal I$.\end{ax}

Similarly, $\mathrm{PID}^w_\kappa$ fails in (2).
\bigskip

The proof of the next lemma  can be found in \cite[page 578]{J}.

\begin{lemma} \label{gap2}
If $\mathfrak{b} > \aleph_2$ then there is an $(\aleph_2,
\lambda)$-gap for some uncountable regular $\lambda$.
\end{lemma}

As a corollary we get the following strengthening of a  result usually attributed to
Todor\v{c}evi\'{c} for $\mathrm{PID}_{\aleph_2}$  .

\begin{thm} \label{pid_b}
$(\mathrm{PID}_{\aleph_2}^w\wedge \mathrm{PID}_{\aleph_1}^w)$ implies $\mathfrak{b}\leq \aleph_2$.
\end{thm}
\begin{proof}
Suppose that $\mathfrak{b}>\aleph_2$ and nevertheless
$\mathrm{PID}_{\aleph_2}^w$ holds. By Lemma~\ref{gap2} there is an
$(\aleph_2, \lambda)$-gap for some uncountable regular $\lambda$. If
$\lambda=\aleph_1$, then $\mathrm{PID}_{\aleph_1}^w$ fails by
Lemma~\ref{gap1}(1). If $\lambda>\aleph_1$, then
$\mathrm{PID}_{\aleph_2}^w$ fails by Lemma~\ref{gap1}(2).
\end{proof}

\bigskip

Clearly, new ideas for the solution of Problem~2  are needed!
\bigskip

At the opposite extreme from Problem 2, we have:
\medskip

\noindent
{\bf Problem 5.} {\it Can $\b$ be ``arbitrarily large'' and still
be  the minimum cardinality of
a locally compact, $\w_1$-compact space that is not
$\s$-countably compact? of one that is also normal?}
\medskip

The question of what happens under Martin's Axiom $(MA)$ is
especially interesting since it implies $\b = \frak c$. It also
implies that $\clubsuit$ fails and that  there are no Souslin trees.
Now a Souslin tree with the interval topology is of cardinality
$\A_1$, and is locally compact, locally countable, $\w_1$-compact
and hereditarily collectionwise normal \cite[4.18]{N2} (hence
hereditarily normal and hereditarily strongly cwH) but is not
$\s$-countably compact.
\medskip

 To show
$\w_1$-compactness, use the fact that every closed discrete subspace in a tree
with the interval topology is a countable union of antichains, and the fact that
a Souslin tree is, by definition, an uncountable tree in which every chain
and antichain is countable. To show that a Souslin tree is not
the union of countably many countably compact subspaces, use these facts
together with the fact that at least one of these subspaces would have to be uncountable,
and the observation that the Erd\H os - Rad\'o theorem implies that
every uncountable tree must either contain an
uncountable chain, or an  infinite antichain, which would be a closed discrete subspace,
contradicting countable compactness. [The interval topology on a subtree is not
always the relative topology, but the relative topology is finer, so
new infinite closed discrete subspaces could arise.]
\medskip

It is worth noting here that adding a single Cohen real
adds a Souslin tree but leaves $\b$
as the same aleph that it is in the ground model. So $\b$ can
be arbitrarily large while the minimum cardinality for Problem 1
is $\A_1$.
\bigskip

\section{A maximum cardinality theme}

Since most of the examples mentioned are locally countable, it is
natural to inquire what happens if local countability is added
to local compactness and $\w_1$-compactness.
The following very simple problem is a dramatic counterpoint to
Question 1.1.
\medskip

\noindent
{\bf Problem 6.} Is there a ZFC example of a  normal, locally countable,
countably compact space of cardinality greater than $\A_1$?
\medskip

Although this problem does not mention local compactness, that
is easily implied by normality (indeed regularity), local
countability, and countable compactness.
\bigskip

In a paper in preparation, \cite{Ns}, the first author shows
 that there does exist a consistent example of a locally countable,
normal, $\w$-bounded (hence countably compact) space
of cardinality $\A_2$, under the axiom $\square_{\A_1}$.
This axiom is consistent if ZFC is consistent, and the
equiconsistency strength of its negation is that of a Mahlo cardinal
\cite[Exercise 27.2]{J}. But the following question, a counterpoint
to Problem 2, is completely open --- no consistency results are
known either way:
\medskip

\noindent
{\bf Problem 7.} Is there a normal, locally countable,
countably compact space of cardinality greater than $\A_2$?
\medskip

Of course, one could also ask whether there is an upper bound
on the cardinalities of normal, locally countable,
countably compact spaces. If ``countably compact'' is strengthened
to ``$\w$-bounded'' then it is consistent, modulo large cardinals,
that there is an upper bound. In fact, it has long been known that every regular,
locally countable, $\w$-bounded space is of cardinality $< \A_\w$ if
the Chang Conjecture variant $(\A_{\w+1}, \A_\w) \to (\A_1, \A_0)$
holds \cite{JSS}. It has recently been shown by Monroe Eskew and Yair Hayut
\cite{EH} that the consistency of a huge cardinal implies
that this variant is consistent. Previously, the best large cardinal axiom known was the existence of a 2-huge cardinal.
\medskip

On the other hand, it has long been known that there are examples of regular,
locally countable, $\w$-bounded spaces of cardinality $\A_n$ for all finite $n$
 just from ZFC, as
was shown in  \cite{JNW}. However, the construction in \cite{JNW} is so general
that it could even produce non-normal
closed subspaces of cardinality $\A_1$, and it is difficult to see how these
can be avoided in ZFC without many additional details, if at all.
\bigskip

We conjecture that it is consistent, modulo large cardinals,
that every normal, locally  countable, countably compact space is
of cardinality $< \A_\w$, leaving only the consistency of examples
of cardinality $\A_n$ unknown for $3 \le n < \w$ . This is because
of the  Chang Conjecture variant result \cite{JSS} for the $\w$-bounded case and the following
lemma:

\begin{lemma} \label{conj} The PFA implies every normal, first countable,
countably compact space is $\w$-bounded. \end{lemma}

\begin{proof} This follows from the following  three statements,
the first two being consequences of the PFA:
\smallskip

(1) every first countable,
countably compact space is either compact or
contains a copy of $\w_1$ \cite[6.5, 6.6]{D},
\smallskip

(2)  $\mathfrak p = \mathfrak c = \A_2$, and:
\smallskip

(3)  \cite[3.9]{H} $\mathfrak p = \mathfrak c$ is equivalent to the following statement.
\medskip

\noindent
 {\bf H}: {\it If $D$ is a countable dense subset of a $T_1$ space $K$, and there exists an open cover $\mathcal U$ of $K$ such that
$\{U \cap D: U \in \mathcal U \}$ has cardinality $< \mathfrak c$, and admits no finite subcover of $D$, then $D$ has an
infinite closed discrete subset.}

\medskip
 Let $X$ be a normal, first countable, countably compact space with a countable dense subset $D_0$.  If $X$ is not compact,
then by (1), identify $\w_1$ with a subspace of $X$. By first countability, $\w_1$ is closed in $X$.
For each $\a \in \w_1$,
let $V_\a$ be an open neighborhood of $[0, \a]$ in $X$ whose closure misses $\w_1 \sm [0, \a]$. This does not require normality,
only the elementary fact that in a Hausdorff space, a compact set and a closed set disjoint from it can be put into disjoint open sets.
\medskip

Let $\mathcal V = \{V_\a : \a \in \w_1\}$ and let $V = \bigcup \mathcal V$. Assuming normality, let $G$ be an open neighborhood of
$\w_1$ whose closure $K$ is a subset of $V$ and let $D = D_0 \cap G.$ Then $D$, $K$, and $\mathcal U = \{V_\a \cap K: \a \in \w_1\}$
are as in statement {\bf H}, so in $K$ there is an infinite closed discrete subset of $D$. But $K$ is a closed subset of $X$, contradicting countable compactness of $X$.
\end{proof}

\begin{rem}  In the light of Lemma \ref{conj}, the questions in this section have natural generalizations
for first countable, locally compact spaces. The key is that every point in such a space either has a
a countable neighborhood, or a compact neighborhood of cardinality $\mathfrak c$. \end{rem}

For instance, our conjecture can be rephrased:

\begin{Conj} It is consistent, modulo large cardinals, that every normal, locally compact, first countable,
countably compact space is of Lindel\"of degree $< \A_\w$ and hence
of cardinality $< max\{\A_\w, \mathfrak c^+\}$.\end{Conj}

If ``normal" is replaced by ``locally hereditarily Lindel\"of," the answer is Yes: the Chang Conjecture variant does it \cite{JSS}.

\section{Locally compact, quasi-perfect preimages}

Example \ref{minb} was
the first nontrivial example for the following problem
by van Douwen \cite{vD2}.
\medskip

\noindent
{\bf Problem 8.} {\it Is ZFC enough to imply that each
 first countable
regular space of cardinality at
most $\frak c$ is a quasi-perfect image of some locally compact space? }

\begin{defn} \rm A continuous map $: X \to Y$ is
{\it quasi-perfect} if it is surjective, and closed, and
each fiber $f^\from \{x\}$ is countably compact.\end{defn}

The following theorem was the key to Example \ref{minb}.

\begin{thm} \label{E} \cite{N7} Let $E$ be a stationary, co-stationary
subset of $\w_1$. There is a locally compact, normal, quasi-perfect preimage
of $E$, of cardinality $\frak b$. If $\frak b = \A_1$, then this preimage can
also be locally countable, hence first countable.\end{thm}

The stated  properties of Example \ref{minb} follow from this theorem and from the
following simple facts.

\begin{thm} \label{maps} Let $f : Y \to X$ be a continuous surjective map.
\medskip

\noindent
$(i)$ If $X$ is not $\s$-countably compact, neither is $Y$.
\medskip

\noindent
$(ii)$ If $X$ is $\w_1$-compact, and $f$ is closed, and each
fiber $f^\from \{x\}$ is $\w_1$-compact, then $Y$ is $\w_1$-compact.

\end{thm}

\begin{proof} Statement $(i)$ easily follows by contrapositive and the
elementary fact that the continuous image of a countably compact
space is countably compact.
\medskip

To show $(ii)$, let $A$ be
an uncountable subset of $Y$. Then either $A$ meets some fiber in
an uncountable subset, in which case it is not closed discrete
in $Y$, or $f^\to[A]$ is uncountable and so it is not closed discrete.
Let $B$ be a subset of  $f^\to[A]$  that is not closed,
 and let $p \in \ob{B}\setminus B$.
\smallskip

Let $A_0 = \{x_b : b \in B\}$ be a subset of $A$ such that
$f(x_b) = b$ for all $b \in B$. Because $f$ is closed,
and $B$ is not closed, neither is $A_0$ and, in fact,
it has an accumulation point in  $f^\from \{p\}$.
\end{proof}

\medskip
\def \P{\Bbb P}

If $E$ is co-stationary, then all countably compact subsets of
$E$, being closed, are countable. So $E$ cannot be $\s$-countably compact
unless it is countable. And if $E$ is stationary, it has limit points in
the closure of every uncountable subset of $\w_1$, and so it is $\w_1$-compact.
It therefore follows from Theorem \ref{maps} that any quasi-perfect preimage of
$E$ is $\w_1$-compact, but not $\s$-countably compact.
\bigskip

Problem 8 was motivated by a theorem in 13.4 of \cite{vD2},
which stated
that the preimages it asks for do exist if $\frak b = \frak c$.
The preimages van Douwen  constructed were locally countable. But it is
also consistent that some of them are not normal:

\begin{thm} \label{P} Let $X$ be a locally compact, locally countable,
quasi-perfect preimage of the space $\P$ of irrationals with
the relative topology.
If the PFA holds, then $X$ is not normal.\end{thm}

\begin{proof} Let $p \in \P$ and let $\pi: X \to \P$ be a
surjective quasi-perfect map. Let $\{\s_\a : \a < \w_1\}$ be a
family of injective sequences in $X$ with
disjoint $\pi$-images that converge to $p$.  Let $A_\a$ be the set of all
limit points of $\s_\a$. Clearly $A_\a \subset \pi^\from \{p\}$,
so $A_\a \cup ran(\s_\a)$ is a separable, closed,
countably compact subspace of $X$.
\medskip

\noindent{\it Case 1. $A_\a$ is uncountable for some $\a$.}
Then  $ran(\s_\a)$ is a countable subspace of $A_\a \cup ran(\s_\a)$
 that does not have compact closure, but now Lemma \ref{conj} shows
that $A_\a \cup ran(\s_\a)$, and hence $X$, is not normal.
\medskip

\noindent{\it Case 2. $A_\a$ is countable for all $\a$.}
By induction, build sequences $\tau_\a (\a < \w_1)$ in $X$
whose $\pi$-images
converge to $p$, and  such that $ran(\s_\beta) \subset^* ran(\tau_\a)$
for all $\beta < \a$. If some $\tau_\a$ has uncountably many limit points,
argue as in Case 1. Otherwise, the set $C_\a$ of limit points of each $\tau_\a$
is compact and countable, so
it is contained in a countable, compact, open subset $V_\a$ of $X$.
Let $V =\bigcup_{\a < \w_1} V_\a$.
  Since the $C_\a$ form an increasing chain, their
union $C$ is $\w$-bounded, hence countably compact, hence closed in
$X$ by first countability of $X$. However, inasmuch as
 the $\s_\a$ have disjoint ranges, $\pi [N\setminus C]$ is uncountable for
any neighborhood $N$ of $C$. And so, the $\pi$-image of
every closed neighborhood of $C$
is an uncountable, closed subset of $\P$; hence
it must be of cardinality $\c$. But $|V| = \A_1$, so $V$ cannot
contain a closed neighborhood of $C$, contradicting normality.
\end{proof}

The italicized consequence of $\frak p > \aleph_1$ in the proof of Lemma
\ref{conj} is
actually equivalent to it. The familiar Franklin-Rajagopalan space
$\gamma \Bbb N$ has a countable dense set of isolated points, and is normal,
locally compact, and locally countable, hence first countable. It
can be made countably compact if, and only if, $\frak p = \aleph_1$.
In \cite{N1a} there is an extended treatment of when $\gamma \Bbb N$
can be hereditarily normal. Theorem \ref{Tri3} excludes $MM(S)[S]$ models,
but we do not know whether these models can be substituted for the
PFA in Theorem \ref{P}. In these models, $\frak p = \aleph_1$ while
$\frak b = \c = \aleph_2$.

\bigskip

In the absence of $\frak  b = \frak c$, we still have very little
idea of which first countable spaces have locally compact
quasi-perfect preimages. It is even unknown whether there is a model
of $\b < \c$ in which  the space
$\Bbb P$ of irrationals has such a preimage, or whether there
is a model in which it does not have one. It is also unknown
whether there is a  locally compact, locally countable, quasi-perfect
preimage of $[0,1]$ if $\b < \c$. Such a space would solve the following problem:
\medskip

\noindent
{\bf Problem 9.} Is there a ZFC example of a scattered,
countably compact, regular space that can be mapped
continuously onto $[0,1]$?
\medskip

See \cite{N5} for discussion of this problem, including an
explanation why the answer is affirmative if ``regular''
is omitted.

\end{document}